\documentclass[12pt, A4,reqno]{amsart}
\usepackage[utf8]{inputenc}
\usepackage[english, french, english]{babel}
\usepackage{amsmath}
\usepackage{amssymb}
\usepackage{amsfonts,color}
\usepackage{dsfont}
\numberwithin{equation}{section}
\usepackage{a4wide}
\usepackage{graphicx}
\usepackage{comment}
\usepackage{accents}
\usepackage{pgf,tikz}
\usepackage{mathrsfs}
\usepackage[colorlinks=true, linkcolor=black, urlcolor=blue, filecolor=black, citecolor=black, menucolor=black]{hyperref}
\usetikzlibrary{arrows}

\newcommand{\CC}{\mathbb{C}}
\newcommand{\DD}{\mathbb{D}}
\newcommand{\NN}{\mathbb{N}}
\newcommand{\RR}{\mathbb{R}}

\newcommand{\cF}{{\mathcal{F}}}

\renewcommand{\tilde}{\widetilde}

\newcommand{\bea}{\begin{eqnarray}}
\newcommand{\eea}{\end{eqnarray}}

\newcommand{\beqa}{\begin{eqnarray*}}
\newcommand{\eeqa}{\end{eqnarray*}}

\newcommand{\lra}{\longrightarrow}

\DeclareMathOperator{\Hol}{Hol}

\DeclareMathSymbol{\subsetneqq}{\mathbin}{AMSb}{36}

\DeclareMathOperator{\Pol}{Pol}

\DeclareMathOperator{\diam}{diam}

\newtheorem{theorem}{Theorem}[section]
\newtheorem{lem}[theorem]{{\bf Lemma}}
\newtheorem{coro}[theorem]{{\bf Corollary}}
\newtheorem{prop}[theorem]{{Proposition}}
\newtheorem{remark}[theorem]{{Remark}}

\title{Sampling constants in generalized Fock spaces}

\author{S. Konate \& M.-A. Orsoni}

\thanks{The research of the first author is partially supported by Banque Mondiale via the Projet d'Appui
au D\'eveloppement de l'Enseignement Sup\'erieur du Mali. The research of the second author is partially supported by the project ANR-18-CE40-0035 and by the Joint French-Russian Research Project PRC CNRS/RFBR 2017--2019.}

\address{University of Segou, Mali}
\email{gnatiosia@gmail.com}

\address{Univ. Bordeaux, CNRS, Bordeaux INP, IMB, UMR 5251,  F-33400, Talence, France}  
\email{Marcu-Antone.Orsoni@math.u-bordeaux.fr}

\begin{document}
\begin{abstract}
We prove several results related to a Logvinenko-Sereda type theorem on dominating sets for generalized doubling Fock spaces. 
In particular, we give a precise polynomial dependence of the sampling constant on the relative density parameter $\gamma$ of the dominating set. Our method is an adaptation of that used in \cite{HKO} for the Bergman spaces and is based on a Remez-type inequality and a covering lemma related to doubling measures. 
\end{abstract}
\maketitle

\section{Introduction}
Sampling problems are central in signal theory and cover, for instance, sampling
sequences and so-called dominating sets which allow to recover the norm of a signal --- defined
by an integration over a given domain --- from the integration on a subdomain (precise definitions
will be given later).
They have been considered in a large variety of situations (see e.g. the survey \cite{FHR}), including the Fock space and its generalized versions (see \cite{Se, MMO, OC, Li, JPR, LZ}). 
In this paper we will focus on the second class of problems, i.e. dominating sets.
Conditions guaranteeing that a set is dominating were established rather long ago (in the 70's for the Paley-Wiener space and in the 80's for
the Bergman and Fock spaces, see e.g. the survey \cite{FHR} and references therein).
More recently, people got interested in estimates of the sampling constants which give quantitative information on the
tradeoff between the cost of the sampling and the precision of the estimates. The major paper
on this connection is by Kovrijkine \cite{Ko} who gave a method to consider this problem in the
Paley-Wiener space.
Subsequently, his method was adapted to several other spaces (see \cite{HJK, JS, HKO, GW, BJP}). 
In this paper, inspired by the methods in \cite{HKO},  we will discuss the case of generalized doubling
Fock spaces. For these, the paper \cite{MMO} provides a wealth of results that 
allow to translate the main steps of \cite{HKO} to this new framework. \\

{\bf Notations.} As usual, $A \lesssim B$ (respectively $A \gtrsim B$) means that there exists a constant $c>0 $ independent of the relevant variables such that $A \le c B$ (resp. $A \ge c B$). Similarly, $A \asymp B$ stands for $A \lesssim B$ and $A \gtrsim B$.

\subsection{Doubling subharmonic functions and generalized Fock spaces}
A function $\phi:\CC\lra \RR$ of class $C^2$ is said to be \emph{subharmonic} if $\Delta \phi \geq 0$, and \emph{doubling subharmonic} if it is subharmonic and the (non-negative) measure $\mu:=\Delta\phi$ is
doubling, i.e. there exists a constant $C$ such that for every $z\in\CC$ and $r>0$, 
\begin{equation}
\label{eq-doubling-measure}
 \mu(D(z,2r))\le C \mu(D(z,r)).
\end{equation}
Here $D(z,r)$ denotes the standard open euclidean disk of center $z\in \CC$ and radius $r >0$.
The minimal constant satisfying \eqref{eq-doubling-measure} is denoted by $C_\mu$ and called \emph{doubling constant} for $\mu$. 
A basic example of doubling subharmonic function is $\phi(z)=|z|^2$ for which $\mu=\Delta \phi$ is equal to $4$ times the Lebesgue measure and $C_\mu=4$. The reader may consider this simpler case for the time of the introduction.

Doubling subharmonic functions induce a new metric on the complex plane $\CC$ which is more adapted to the complex analysis we shall work with. We postpone the advanced results on doubling measures to Section \ref{sec-reminders} but let us introduce some basic objects related to it.  
First, we can associate with $\phi$ as above a function $\rho:\CC\lra \RR^+$, such that
\[
 \mu(D(z,\rho(z)))=1.
\]
To get an idea, assuming $\phi$ is suitably regularized, we have $\Delta\phi \asymp \rho^{-2}$ (see \cite[Theorem 14 and Remark 5]{MMO}).
Next, denote $D^r(z):=D(z,r\rho(z))$ the disk adapted to the new metric and write $D(z):=D^1(z)$ the disk with unit mass for the measure $\mu$. Finally, with these definitions in mind we can introduce the following natural density. A measurable set $E$ is \emph{$(\gamma,r)$-dense} (with respect to the above metric), if for every $z \in \CC$
\[
 \frac{|E\cap D^r(z)|}{|D^r(z)|}\ge \gamma. 
\]
Here $|F|$ denotes planar Lebesgue measure of a measurable set $F$. 
We will just say that the set is \emph{relatively dense} if there is some $\gamma>0$ and some
$r>0$ such that the set is $(\gamma,r)$-dense. We shall see in the next subsection that relative density is the right notion to characterize dominating sets.

Let $dA$ be the planar Lebesgue measure on $\CC$. For a doubling subharmonic function $\phi:\CC\lra \RR$ and $1 \le p<+\infty$, we will denote by $L^p_{\phi}(F)$ the Lebesgue space on a measurable set $F
\subset \CC$ with respect to the measure $e^{-p\phi}dA$. We will also use the notation 
$$\|f\|^p_{L^p_{\phi}(F)}:=\int_{F}|f|^pe^{-p\phi}dA.$$

Finally, we define the \emph{doubling Fock space}, by
\[
 \cF^p_{\phi}=\{f\in \Hol(\CC):\|f\|_{L^p_{\phi}(\CC)}^p:=\int_{\CC}|f|^pe^{-p\phi}dA<+\infty\}.
\] 
These spaces appear naturally in the study of the Cauchy-Riemann equation (see \cite{Christ1991, COC}) and are well studied objects (see e.g. \cite{OP2016} for Toeplitz operators on these spaces). When $\phi(z)=|z|^2$, $\cF^p_{\phi}$ is the classical Fock space (see the textbook \cite{Zhu} for more details).

\subsection{Dominating sets and sampling constants}
A measurable set $E\subset \CC$ will be called \emph{dominating} for $\cF^p_{\phi}$ if there exists $C>0$ such that
\bea\label{dom}
 \int_E|f|^pe^{-p\phi}dA\ge C^p \int_{\CC}|f|^p e^{-p\phi}dA, \quad \forall f \in \cF^p_{\phi}.
\eea
The biggest constant $C$ appearing in \eqref{dom} is called the \emph{sampling constant}. 

The notion of dominating sets can be defined for several spaces of analytic functions and their characterization has been at the center of intensive research starting with a famous result of Panejah \cite{Pa1, Pa2}, and Logvinenko and Sereda \cite{LS} for the Paley-Wiener space, which consists of all functions in $L^2$ whose Fourier transform is supported on a fixed compact interval. For corresponding results in the Bergman space we refer to \cite{Lu1, Lu2, Lu3, GW}. As for Fock spaces, this has been done by Jansen, Peetre and Rochberg in \cite{JPR} for the classical case $\varphi(z)=|z|^2$ and by Lou and Zhuo in \cite[Theorem A]{LZ} for generalized doubling Fock spaces (see also \cite{OC} for a characterization of general sampling measures). 
It turns out that a set $E$ is dominating for the doubling Fock spaces if and only if it is relatively dense. This characterization also holds for the other spaces of analytic functions mentioned above with an adapted notion of relative density.

Once the dominating sets have been characterized in terms of relative density, a second question of interest is to know whether the sampling constants can be estimated in terms of the density parameters $(\gamma, r)$. Kovrijkine answered this question in \cite{Ko} giving a sharp estimate on the sampling constants for the Paley-Wiener spaces, with a polynomial dependence on $\gamma$ (see also \cite{Rez2010} for sharp constants in some particular geometric settings). Such a dependence is used for example in control theory (see \cite{EV, ES, BJP} and the references therein). Kovrijkine's method involves Remez-type and Bernstein inequalities. It has been used in several spaces in which the Bernstein inequality holds (see for instance \cite{HJK} for the model space, \cite{BJP} for spaces spanned by Hermite functions, or \cite{ES} for general spectral subspaces), and also to settings where a Bernstein inequality is not at
hand (e.g. Fock and polyanalytic Fock spaces \cite{JS} or Bergman spaces \cite{HKO, GW}). In \cite{HKO}, the authors developed a machinery based on a covering argument to circumvent Bernstein's inequality and the aim of this paper is to adapt this method to doubling Fock spaces.

\subsection{Main results}
It can be deduced from Lou and Zhuo's result \cite[Theorem A]{LZ} that if $E$ is $(\gamma, r)$-dense then there exist some constants $0<\varepsilon_0 <1$ and $c >0$ depending only on $r$ such that inequality \eqref{dom} holds for every 
$$C^p \le c \gamma \varepsilon_0^{\frac{2(p+2)}{\gamma}}.$$
This last estimate gives an exponential dependence on $\gamma$. In the spirit of the work \cite{Ko}, we improve it providing a polynomial estimate in $\gamma$ with a power depending suitably on $r$. 
\begin{theorem}\label{thm1}
Let $\phi$ be a doubling subharmonic function and $1\le p<+\infty$. 
Given $r>1$, 
there exists $L$ such that for every measurable set 
$E\subset \CC$
which is $(\gamma,r)$-dense, we have
\bea\label{estim1}
  \|f\|_{L^p_{\phi}(E)} \ge \left(\frac{\gamma}{c}\right)^L \|f\|_{L^p_{\phi}(\CC)}
\eea
for every $f\in \cF^p_{\phi}$.
Here, the constants $c$ and $L$ depend on $r$,
and for $L$ we can choose
\begin{equation*}
 L
\lesssim
r^{\log_2(C_\mu)} + \frac{1}{p}(1 +  \log(r))
\end{equation*}
where $C_\mu$ is the doubling constant (see \eqref{eq-doubling-measure}) and the implicit constant depends only on the space (i.e.  only on $\phi$). 
\end{theorem}
\begin{remark}
When $\phi(z)=|z|^2$, we have $\log_2(C_\mu)=2$ and we get the same result as that given in \cite[Theorem 4.6]{JS} for the classical Fock space with an explicit dependence on $p$ in addition.
\end{remark}

Observe that we are mainly interested in the case when $r$ is big, for instance $r\ge 1$ (in case $E$ is $(\gamma,r)$-dense for some $r<1$ we can also show that $E$ is $(\tilde{\gamma},1)$-dense for some $\tilde{\gamma} \asymp \gamma$ where underlying constants are universal).\\

The proof of Theorem \ref{thm1} follows the scheme presented in \cite{HKO}. We will recall the necessary results
from that paper. The main new ingredients come from \cite{MMO} and concern a finite overlapping property
and a lemma allowing us to express the subharmonic weight 
locally as (the real part of) a holomorphic function.

We should mention that in \cite[Theorem 7]{LZ}, in the course of proving that the relative density is a necessary condition for domination, it is shown that $\gamma \gtrsim C^p$. Hence, we cannot expect better than a polynomial dependence in $\gamma$ of the sampling constant $C$. In this sense, our result is optimal. \\

As noticed in a remark in \cite[p.11]{Lu1} and after Theorem 2 in \cite{HKO} for the Bergman space, there is no reason 
a priori why a holomorphic function for which the integral $\int_E|f|^pe^{-p\phi}dA$ 
is bounded for a relative dense set $E$ should be in $\cF^p_{\phi}$. 
Outside the class $\cF^p_{\phi}$ relative density is 
in general not necessary for domination. For this reason, we always assume that we test on functions $f \in \cF^p_{\phi}$. 
\\

As a direct consequence of Theorem \ref{thm1}, we obtain a bound for the norm of the inverse of a Toeplitz operator $T_\varphi$. We remind that for any bounded measurable function $\varphi$, the Toeplitz operator $T_\varphi$ is defined on $ \cF^2_{\phi}$ by $T_\varphi f = {\bf{P}}(\varphi f)$ where $\bf{P}$ denotes the orthogonal projection from $L^2_\phi(\CC)$ onto $ \cF^2_{\phi}$. As remarked in \cite[Theorem B]{LZ}, for a non-negative function $\varphi$, $T_\varphi$ is invertible if and only if  $E_s = \left\{z \in \mathbb{C} :  \varphi(z) > s \right\}$ is a dominating set for some $s>0$. Tracking the constants we obtain 
\begin{coro}
\label{cor1}
Let $\varphi$ be a non-negative bounded measurable function. 
The operator $T_\varphi$ is invertible if and only if there exists $s>0$ such that $E_s = \left\{z \in \mathbb{C} :  \varphi(z) > s \right\}$ is $(\gamma, r)$-dense for some $\gamma > 0$ and $r >0$. 
In this case, we have 
$$||T_\varphi^{-1}|| \leq \frac{\| \varphi\|^{-1}_{\infty}}{1-\sqrt{1-\left(\frac{s}{\|\varphi\|_{\infty}}\right)^2 \left(\frac{\gamma}{c}\right)^{2L}}} .$$ 
\end{coro}
Notice that the right-hand side behaves as $\gamma^{-2L}$ as $\gamma \to 0$.
For the sake of completeness, we will give a proof of the reverse implication at the end of Section \ref{proof}.\\

The paper is organized as follows. In Section \ref{sec-reminders} we recall several results from \cite{MMO} concerning doubling measures and subharmonic functions. In Section \ref{sec-Remez}, we introduce planar Remez-type inequalities which will be a key ingredient in the proof of Theorem \ref{thm1}. Finally, we prove the covering lemma and Theorem \ref{thm1} in Section \ref{proof} and deduce Corollary \ref{cor1}.

\section{Reminders on doubling measures \label{sec-reminders}}
A non-negative doubling measure $\mu$ (see \eqref{eq-doubling-measure} for the definition) determines a metric on $\CC$ to which the usual notions have to be adapted.  
In this section, we recall several geometric measure theoretical tools in connection with doubling measures that we will need later and which come essentially from the paper \cite[Chapter 2]{MMO}. Actually, the results of the present section can be translated in terms of the distance induced by the metric $\rho(z)^{-2} dz \otimes d\bar{z}$ but we will not exploit this point of view here.

We start with a standard geometric estimate whose main part is due to \cite[Lemma 2.1]{Christ1991}. 
\begin{lem}{\cite[Lemma 1]{MMO}}
\label{lem-1}
Let $\mu$ be a doubling measure on $\CC$. For any disks $D(z, r)$ and $D(z', r')$ such that $r>r'$ and $z' \in D(z, r)$, we have 
$$\frac{1}{2 C_\mu} \left(\frac{r}{r'}\right)^{\kappa}\le \frac{\mu(D(z,r))}{\mu(D(z',r'))}\le C_\mu^2 \left(\frac{r}{r'}\right)^{\log_2(C_\mu)}.$$
where $\kappa=\frac{1}{\lceil C_{\mu}^2 + 2 \rceil}$. 
\end{lem}
Here $\lceil x \rceil$ denotes the smallest integer bigger or equal to $x$ and correspondingly, $\lfloor x \rfloor$ is the biggest integer less or equal to $x$. 
The paper \cite{MMO} refers to \cite{Christ1991} for the proof of this result. We mention that this latter paper only contains the left hand side estimate. For the convenience of the reader and in order to better understand the constants, we provide a detailed proof completing Christ's argument.   
\begin{proof}
Let us start with the right hand side inequality. 
By the triangular inequality, we have $D(z,r) \subset D(z', 2r)$. 
Now, note that 
$$2r= 2^{\log_2\left(\frac{r}{r'}\right)+1} r' \le 2^{\left\lceil \log_2\left(\frac{r}{r'}\right) \right\rceil +1} r'.$$
Hence iterating the doubling inequality \eqref{eq-doubling-measure}, we get 
$$\mu(D(z,r)) \le \mu(D(z', 2r)) \le C_\mu^{\left\lceil \log_2\left(\frac{r}{r'}\right) \right\rceil +1} \mu(D(z',r')) \le C_\mu^2 \left(\frac{r}{r'}\right)^{\log_2\left(C_\mu\right)} \mu(D(z',r')).$$

As mentioned above, the left hand side inequality is given in \cite[Lemma 2.1]{Christ1991}. We reproduce the proof here making the constant $\kappa$ more precise. 
Since $z' \in D(z,r)$, we have $D(z', r) \subset D(z, 2r)$ and so 
\begin{equation}
\label{eq-zz'}
\frac{\mu(D(z,r))}{\mu(D(z',r'))} \ge  \frac{C_\mu^{-1}\mu(D(z,2r))}{\mu(D(z',r'))} \ge \frac{C_\mu^{-1}\mu(D(z',r))}{\mu(D(z',r'))}.
\end{equation}
Hence it suffices to prove that for any $z' \in \CC$,
$$\frac{\mu(D(z',r))}{\mu(D(z',r'))} \ge \frac{1}{2} \left(\frac{r}{r'}\right)^{\kappa}$$
whenever $r >r'$.
Let $k \ge 3$ be an integer to be fixed later. 
First assume that $ \frac{r}{r'}=2^k$.
Then we can construct pairwise disjoint disks $D_1, D_2, \dots D_{k-2}$ such that $D_j$ has radius $2^jr'$, $D(z',r') \subset 3 D_j$ and $D_j \subset D(z',r)$. A way to do that is to choose the disks $D_j$ centered along a radius of $D(z',r)$ and ordered in such a way that the radius increases away from $z'$ (see \bsc{Figure} \ref{fig-Christ}). 

\begin{figure}[h!]
\begin{center}
\includegraphics[scale=0.5]{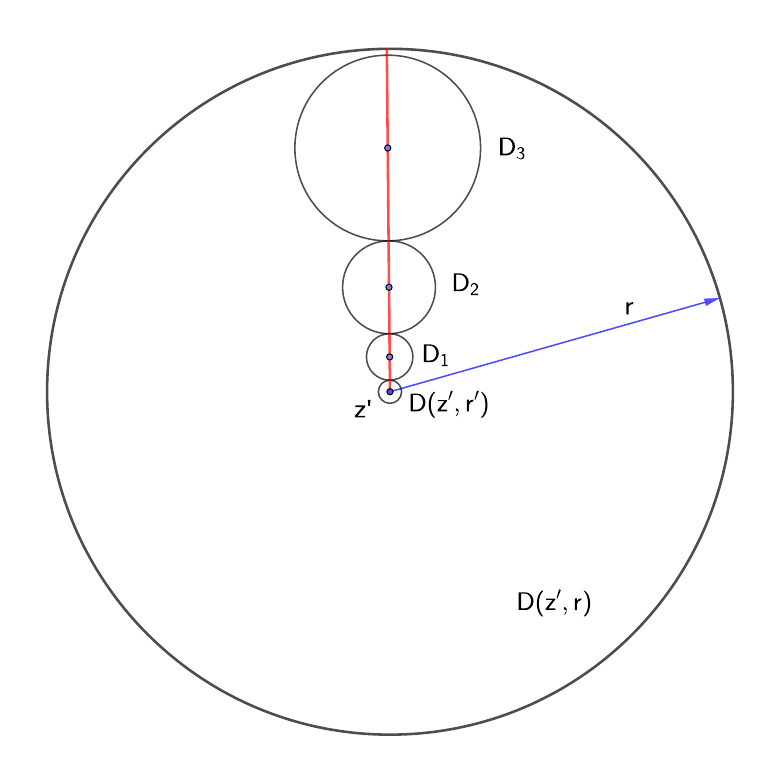}
\end{center}
\caption{The disks $D_j$, $D(z', r')$ and $D(z', r)$. \label{fig-Christ}}
\end{figure}

Now, since $\mu(D(z',r')) \le \mu(3D_j) \le C_\mu^2 \mu(D_j)$, we get 
$$\mu(D(z',r)) \ge \sum_{j=1}^{k-2} \mu(D_j) + \mu(D(z', r')) \ge  [(k-2) C_\mu^{-2} + 1] \mu(D(z',r')). $$
Hence, fixing $k$ to be the smallest integer such that $(k-2) C_\mu^{-2} +1 \ge 2$ i.e $k = \lceil C_\mu^2 + 2 \rceil$, we obtain 
$$\mu(D(z',r)) \ge 2 \mu(D(z',r'))$$
whenever $\frac{r}{r'}=2^k$. 

Let us treat the general case $r >r'$ now.  
If $\frac{r}{r'}\ge 2^k$, we reduce the problem to the previous setting decomposing $D(z',r)$ into a chain of disks $$D(z',r')=D(z',r_0)\subset D(z', r_1) \subset \dots \subset D(z', r_m) \subset D(z',r)$$
where $r_i=2^k r_{i-1}$ for $1\le i \le m$, and $m=\left\lfloor \frac{\log_2(r/r')}{k} \right\rfloor$. 
Hence we can iterate the process and get for every $z' \in \CC$,
\begin{align*}
\mu(D(z',r)) \ge \mu(D(z', r_m)) \ge 2 \mu(D(z', r_{m-1})) \ge \dots &\ge 2^{m} \mu(D(z', r')) \\
&=  2^{\left\lfloor \log_2\left(\frac{r}{r'}\right)/ k \right\rfloor} \mu(D(z',r'))\\
& \ge \frac{1}{2} \left(\frac{r}{r'}\right)^{1/k} \mu(D(z',r')).
\end{align*}
Combined with \eqref{eq-zz'}, this leads to the left hand side inequality of Lemma \ref{lem-1} with 
$$\kappa = \frac{1}{k}= \frac{1}{\lceil C_{\mu}^2 + 2 \rceil}.$$ 

If $\frac{r}{r'}<2^k$, then it is clear that $\frac{1}{2} \left(\frac{r}{r'}\right)^\kappa < 1 \le \frac{\mu(D(z',r))}{\mu(D(z',r'))}$ since $r>r'$. Again, combined with \eqref{eq-zz'}, this yields the expected constant in the left hand side inequality. 
\end{proof}

\begin{remark} 
\label{rmk-constants}
If instead of $z' \in D(z, r)$, we have $D(z,r)\cap D(z',r') \neq \emptyset$ as in \cite[Lemma 1]{MMO}, we get the same inequalities with $C_\mu$ replaced by $C_\mu^2$ in the left hand inequality 
and $C_\mu^2$ replaced by $C_\mu^3$ in the right hand inequality. Also, when $z=z'$, we can remove $C_\mu$ from the left hand side and 
replace $C_\mu^2$ by $C_\mu$ in the right hand side.   
\end{remark}

\begin{remark} 
In the classical Fock space, $\mu$ is the Lebesgue measure (up to a multiplicative constant) and so $C_\mu = 4$. Note that in this case, we obtain an equality with the right hand side and the factor $C_\mu^2$ disappears. This is due to the invariance by translation and the homogeneity of Lebesgue measure. 
\end{remark}

As a direct consequence of Lemma \ref{lem-1} and Remark \ref{rmk-constants}, picking $z'=z$, and replacing $r$ by $r \rho(z)$ and $r'$ by $\rho(z)$, we get the following useful estimate:
for all $z \in \CC$ and $r >1$, 
\begin{equation}
\label{eq-measure}
\frac{1}{2} r^{\kappa} \le 
\mu (D^{r}(z)) \le C_\mu r^{\log_2\left(C_\mu\right)}.
\end{equation}

We are now looking for an upper bound and a lower bound for $\frac{\rho(z)}{\rho(w)}$ whenever $w \in D^r(z)$. 

\begin{lem}
\label{lemma-rho-bounds}
For every $z \in \CC$, we have 
\begin{equation}
\label{eq-estimate-rho}
\forall r >0, \forall w \in D^r(z), \  \frac{\rho(z)}{\rho(w)} \ge \frac{1}{1+r} .
\end{equation}
and 
\begin{equation}
\label{eq-upperbound}
\forall r >1, \forall w \in D^r(z), \  \frac{\rho(z)}{\rho(w)} \lesssim \max\left[ r^{\frac{\log_2(C_\mu)}{\kappa} -1}, 1\right] = r^{\max\left(\frac{\log_2(C_\mu)}{\kappa} -1 , \ 0\right)}.
\end{equation}
\end{lem}

\begin{proof}
A classical fact is that $\rho$ is a $1$-Lipschitz function (see \cite[equation 2.4]{OP2016} for a simple proof): 
$$|\rho(z)-\rho(w)| \le |z-w|, \quad \forall z, w \in \CC.$$
Hence, for every $z \in \CC$, we get inequality \eqref{eq-estimate-rho} by the triangular inequality.
To obtain the upper bound, we use both inequalities of Lemma \ref{lem-1} distinguishing the cases $r \rho(z) \ge \rho(w)$ and $r \rho(z) < \rho(w)$. We get for every $w \in D^r(z)$,
$$ \frac{r \rho(z)}{\rho(w)} \lesssim \max\left[ \mu(D^r(z))^{1/\kappa}, \mu(D^r(z))^{1/\log_2(C_\mu)} \right].$$
Therefore inequality \eqref{eq-measure} implies for $r>1$
$$ \frac{r \rho(z)}{\rho(w)} \lesssim \max \left[r^{\frac{\log_2(C_\mu)}{\kappa}}, r\right],$$
and hence \eqref{eq-upperbound} follows.
\end{proof}
Notice that for $C_\mu$ large enough, we have $\frac{\log_2(C_\mu)}{\kappa} -1 \ge 0$, which means that the maximum in the last inequality \eqref{eq-upperbound} is equal to $r^{\frac{\log_2(C_\mu)}{\kappa} -1}$. This holds exactly when $C_\mu \ge  \sqrt[4]{2}$. 

\bigskip

In the proof of our main theorem, we will need a particular covering of the complex plane. Let us explain how we can construct it. 
We say that a sequence $(a_n)_{n \in \NN}$ is \emph{$\rho$-separated} if there exists $\delta >0$ such that 
$$|a_i - a_j| \geq \delta \max{(\rho(a_i), \rho(a_j))}, \quad \forall i \neq j. $$
This means that the disks $D^\delta (a_n)$ are pairwise disjoint.
We will cover $\CC$ by disks satisfying a finite overlapping property and such that the sequence formed by their centers is $\rho$-separated. 
In \cite{MMO}, the authors construct a decomposition of $\CC$ into
rectangles $R_k$: $\CC=\bigcup_k R_k$, and two such rectangles
can intersect at most along sides. For these rectangles, so-called quasi-squares, there exists a constant $e >1$ depending only on $C_\mu$ such that the ratio of sides of every $R_k$ lies in the interval $[1/e, \, e]$. 
Denoting by $a_k$ the center of $R_k$, Theorem 8(c) of \cite{MMO}
claims in particular that there is $r_0\ge 1$, such that for every $k$,
\[
 \frac{\rho(a_k)}{r_0}\le \diam R_k\le r_0\rho(a_k).
\]
In other words
\bea D^{1/(2C r_0)}(a_k) \subset R_k \subset D^{r_0 /2}(a_k)
\eea
with $C=\sqrt{1+e^2}$. 
As a consequence we get the required covering $\CC=\bigcup D^{r_0/2}(a_k)$ and the sequence $(a_k)$ is $\rho$-separated (take $\delta=1/(2Cr_0)$). 

\bigskip

We end this section by focusing on the particular case of a doubling measure $\mu$ given by the Laplacian $\Delta \phi$ of a subharmonic function $\phi$ which is the case of interest in this paper. We will need to control the value of $\phi$ in a disk by the value at its center. For this, we can use \cite[Lemma 13]{MMO} that we state as follows: for every $\sigma>0$, there exists $A=A(\sigma)>0$ such that for all $k\in \NN$, 
\begin{equation}
\label{estim}
\sup_{z \in D^{\sigma}(a_k)} |\phi(z)-\phi(a_k)-\mathsf{h}_{a_k}(z)| \leq A(\sigma)
\end{equation}
where $\mathsf{h}_{a_k}$ is a harmonic function in $D^\sigma(a_k)$ with $\mathsf{h}_{a_k}(a_k)=0$. Moreover, in view of the proof of \cite[Lemma 13]{MMO} we have 
\begin{equation}
\label{eq-A-sigma}
A(\sigma) \lesssim \sup_{k \in \NN} \mu\left(D^\sigma(a_k)\right) \lesssim \sigma^{\log_2\left(C_\mu\right)}
\end{equation}
where the inequalities come from \cite[Lemma 5(a)]{MMO} and the estimate $\eqref{eq-measure}$. This last result will allow us to translate the subharmonic weight locally into a holomorphic function.

\section{Remez-type inequalities \label{sec-Remez}}
Let us introduce a central result of the paper \cite{AR}.
Let $G$ be a (bounded) domain in $\CC$ and let $0<s<|\overline{G}|$ (Lebesgue measure of $\overline{G}$).
Denoting $\Pol_n$ the space of complex polynomials of degree at most $n\in\NN$, 
we introduce the set
\[
 P_n(\overline{G},s)=\{p\in \Pol_n:|\{z\in \overline{G}:|p(z)|\le 1\}|\ge s\}.
\]
Next, let
\[
 R_n(z,s)=\sup_{p\in P_n(\overline{G},s)}|p(z)|.
\]
This expression gives the biggest possible value at $z$ of a polynomial $p$ of degree at most $n$
and being at most $1$ on a set of measure at least $s$. In particular
Theorem 1 from \cite{AR} claims that for $z\in \partial G$, we have
\bea\label{AR}
 R_n(z,s)\le \left(\frac{c}{s}\right)^n
\eea
where the constant $c$ depends only on the (square of the) diameter of $G$. 
This result corresponds to a generalization to the two-dimensional case of the 
Remez inequality which is usually
given in dimension 1.
In what follows we will essentially consider $G$ to be a disk or a rectangle. 
By the maximum modulus principle, the above constant gives an upper estimate on $G$ for 
an arbitrary polynomial of degree at most $n$ which is bounded by one on a set of
measure at least $s$. Obviously, if this set is small ($s$ close to $0$), i.e. $p$ is controlled by 1 on a 
small set, then the estimate has to get worse.
\\

\begin{remark}\label{rem1}
Let us make another observation. If $c$ is the constant in \eqref{AR} associated with 
the unit disk $G=\DD=D(0,1)$, then a simple argument based on homothety shows that
the corresponding constant for an arbitrary disk $D(0,r)$ is $cr^2$ (considering $D(0,r)$ as underlying domain, the constant $c$ appearing in \cite[Theorem 1]{AR} satisfies {$c>2m_2(D(0,r))$}). So, in the sequel we will use the estimate
\bea
\label{AR1}
 R_n(z,s)\le \left(\frac{cr^2}{s}\right)^n,
\eea
where $c$ does not depend on $r$.
\end{remark}

Up to a translation, the following counterpart of Kovrijkine's result for the planar case has been 
given in \cite{HKO}:
\begin{lem}\label{Kovr-2D}
Let $0<r<R$ be fixed. There exists a constant $\eta>0$ such that the following holds. 
Let $w \in \mathbb{C}$, let 
$E\subset D^r(w)$ be a planar measurable set of positive measure and let $z_0 \in D^r(w)$. For every  $\phi$ analytic in $D^R(w)$, if $|\phi(z_0)|\ge 1$ and 
$M=\sup_{z\in D^R(w)}|\phi(z)|$ then
\[
 \sup_{z\in D^r(w)}|\phi(z)|\le \left( \frac{cr^2\rho(w)^2}{|E|}\right)^{\eta \log M}\sup_{z\in E}|\phi(z)|,
\]
where $c$ does not depend on $r$, and
\[
 \eta \le c''\frac{R^4}{(R-r)^4}\log\frac{R}{R-r}
\]
for an absolute constant $c''$.
\end{lem}

The corresponding case for $p$-norms is deduced exactly as in Kovrijkine's work.

\begin{coro}\label{CoroKov}
Let $0<r<R$ be fixed. There exists a constant $\eta>0$ such that the following holds. 
Let $w \in \mathbb{C}$, let 
$E\subset D^r(w)$ be a planar measurable set of positive measure and let $z_0 \in D^r(w)$. For every  $\phi$ analytic in $D^R(w)$, if $|\phi(z_0)|\ge 1$ and 
$M=\sup_{z\in D^R(w)}|\phi(z)|$ then
for $p\in [1,+\infty)$ we have
\[
 \|\phi\|_{L^p(D^r(w))}\le \left(\frac{cr^2\rho(w)^2}{|E|}\right)^{\eta\log M+\frac{1}{p}}\|\phi\|_{L^p(E)}.
\]
\end{coro}

The estimates on $\eta$ are the same as in the lemma. The constant $c$ does not depend on $r$.

\section{Proof of Theorem \ref{thm1} \label{proof} }
The key ingredient in the proof of Theorem \ref{thm1} and a main contribution of this paper is the existence and the estimate of the covering constant. 
Indeed, we have shown in Section \ref{sec-reminders} that the disks $(D^{s}(a_k))_{k\in \NN}$ cover the complex plane $\CC$ for $s>r_0/2$. Now, we shall prove that there exists a covering constant $N$ depending on the covering radius $s \geq r_0/2$. This means that the number of overlapping disks cannot exceed $N$. We denote by $\chi_F$ 
the characteristic function of a measurable set $F$ in $\DD$.
We are now in a position to prove the covering lemma. 
\begin{lem}\label{lem1}
For $s > 1$, there exists a constant $N:=N(s)$ such that
\[
 \sum_{k}\chi_{D^s(a_k)}\le N.
\]
Moreover, there exists some universal constant $c_{ov}:=c_{ov}(\phi) >0$ such that
\[
N \le c_{ov} s^{\alpha}
\]
where 
$$\alpha = \min \left[2\max\left(1+\frac{\log_2(C_\mu)}{\kappa}, \, 2\right); \ \max\left(\frac{\log_2(C_\mu)}{\kappa}, \, 1\right) \log_2(C_\mu)\right],$$
$C_\mu$ is the doubling constant and $\kappa$ is given in Lemma \ref{lem-1}.
\end{lem}

Obviously, the constant $N$ is at least equal to 1 and since $N$ is non-decreasing in $s$, we are only interested in the behavior when $s$ goes to $\infty$. We also remind the reader that in order to cover the complex plane $\CC$, we need to require that $s \ge r_0/2$. 

\begin{proof}[Proof of Lemma \ref{lem1}]
The proof is inspired by that of \cite[Lemma 2.5]{AS2011}. 
Let $s >1$ and $z \in \CC$. Denote $\Gamma(z)=\left\{ k \in \NN, \ z \in D^s(a_k) \right\}$. Our goal is to estimate $\#\Gamma(z)$. 
Write $\theta = \frac{\log_2(C_\mu)}{\kappa}$. 
Recall that for $z \in D^s(a_k)$, we have by \eqref{eq-estimate-rho} and \eqref{eq-upperbound}
\begin{equation}
\label{eq-estimates-for-rhoak}
\frac{\rho(z)}{1+s}\le\rho(a_k) \lesssim \rho(z) \max(s^{\theta -1}, 1).
\end{equation}
Now, pick $\delta >0$ such that the disks $D^{\delta}(a_n)$ are disjoint (recall that $(a_n)_{n \in \NN}$ is $\rho$-separated). By the triangular inequality, there exists a constant $C>0$ such that $D^{\delta}(a_k) \subset D^{C \max(s^\theta, s)}(z)$ whenever $z \in D^s(a_k)$ (see Figure \ref{Figure}).
Indeed, if $w \in D^\delta(a_k)$ then
\begin{align*}
|w-z| \le |w-a_k| + |z-a_k| &\le \delta \rho(a_k) + s \rho(a_k) \\
&\lesssim (\delta + s) \rho(z) \max(s^{\theta -1}, 1) \\
&\lesssim \rho(z) \max(s^\theta, s),
\end{align*}
where we have used that $\delta >0$ is fixed and $s >1$.

\begin{figure}[h!]
\begin{center}
\includegraphics[scale=0.5]{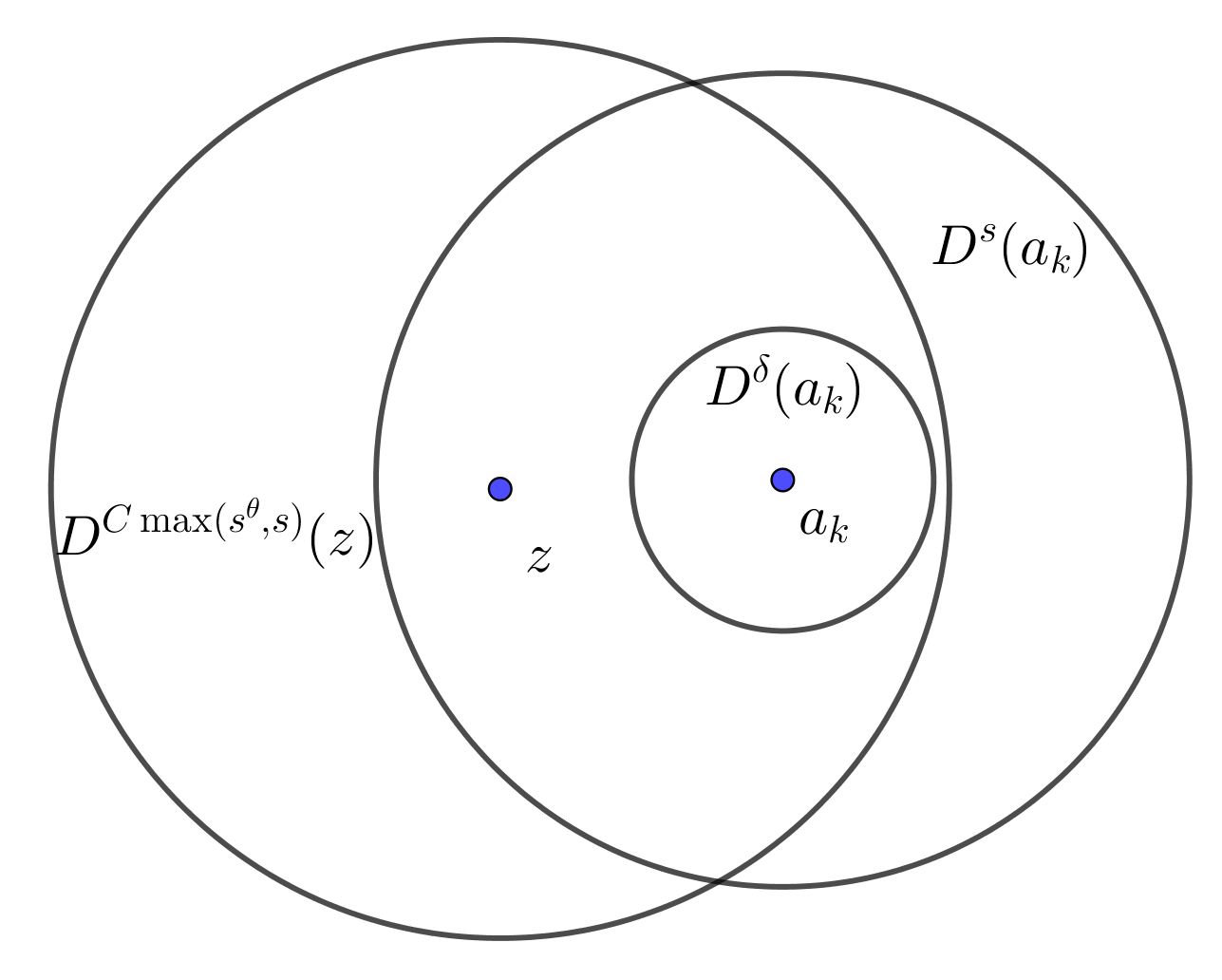}
\end{center}
\caption{The disks $D^s(a_k)$, $D^{C\max(s^\theta, s)}(z)$ and $D^\delta(a_k)$. \label{Figure}}
\end{figure}

Hence, since $D^\delta(a_k)$ are disjoint and with \eqref{eq-estimates-for-rhoak} in mind, it follows 
\begin{align*}
 \# \Gamma(z) &\le \# \left\{ k \in \NN, \ D^{\delta}(a_k) \subset D^{C\max(s^\theta, s)}(z)\right\}\\
 &\le \frac{\left| D^{C\max(s^\theta, s)}(z)\right|}{\inf_{k \in \Gamma(z)} \left|D^{\delta}(a_k)\right|}\\ 
 &= \sup_{k \in \Gamma(z)} \left( \frac{\left[C \max(s^\theta, s) \rho(z)\right]^2}{(\delta \rho(a_k))^2}\right) \\
 & \lesssim \left[\max(s^\theta, s) (1+s)\right]^2\\
 & \lesssim s^{2\max\left(1+\theta, \, 2\right)} 
 \end{align*} 
 where we have used that $\max(s^\theta, s) = s^{\max(\theta, 1)}$ for $s>1$.
 
Moreover, notice that $\mu\left[D^{\delta}(a_k)\right] \gtrsim 1$ since $0 <\delta <1$ is fixed. Hence an analogous computation, replacing the Lebesgue measure by the measure $\mu$ and using the right hand side of inequality \eqref{eq-measure}, leads to 
 \begin{align*}
 \# \Gamma(z) &\le \# \left\{ k \in \NN, \ D^{\delta}(a_k) \subset D^{C\max(s^\theta, s)}(z)\right\}\\
 &\le \frac{ \mu \left[D^{C\max(s^\theta, s)}(z)\right]}{\inf_{k \in \Gamma(z)} \mu\left[D^{\delta}(a_k)\right]}\\
&\lesssim \max(s^\theta, s)^{\log_2(C_\mu)}\\
&\lesssim s^{\max(\theta, \, 1) \log_2(C_\mu)}.
 \end{align*}
 Finally, taking the minimum between the two estimates, we get for $s>1$
 $$\# \Gamma(z) \lesssim \min\left(s^{2\max\left(1+\theta, \, 2\right)}, \, s^{\max(\theta, \, 1) \log_2(C_\mu)} \right) = s^{\min \left[2\max\left(1+\theta, \, 2\right), \, \max(\theta, \, 1) \log_2(C_\mu)\right]}.$$
Since the estimates are uniform in $z$, the lemma follows. 
\end{proof}

As in \cite{HKO} for the Bergman space, we now introduce good disks.
Fix $r_0/2 \le s<t$ where $r_0/2$ is the radius such that the disks $D^{r_0/2}(a_k)$ cover the complex plane. 
For $K>1$ the set
\[
 I_f^{K-good}=\{k:\|f\|_{L^{p}_{\phi}(D^{t}(a_k))}\le K \|f\|_{L^p_{\phi}(D^s(a_k))}\}
\]
will be called the set of $K$-good disks for $(t,s)$
(in order to keep notation light we will not include
$s$ and $t$ as indices). This set depends on $f$.
\\

The following proposition has been shown in \cite{HKO} for the Bergman space, but
its proof, implying essentially the finite overlapping property, is exactly the same.

\begin{prop}\label{prop1}
Let $r_0/2 \le s<t$.
For every constant $c\in (0,1)$, there exists $K$ such that for every $f\in \cF^p_{\phi}$ we have
\[
 \sum_{k\in I_f^{K-good}} \|f\|_{L^p_{\phi}(D^s(a_k))}^p\ge c \|f\|_{L^p_{\phi}(\CC)}^p.
\]
\end{prop}
One can pick $K^p = N(t)/(1-c)$ where $N(t)$ corresponds 
to the overlapping constant from Lemma \ref{lem1} for the radius $t$.
\\

We are now in a position to prove the theorem.
\begin{proof}[Proof of Theorem \ref{thm1}]
Take $s=\max(r, \, r_0/2)$ and $t=4s$. As noted in \cite{HKO}, if $E$ is $(\gamma, r)$-dense then $E$ is $(\widetilde{\gamma}, s)$-dense, where $\tilde{\gamma}= c \gamma$ and $c$ is a multiplicative constant. Hence we can assume that $E$ is $(\gamma, s)$-dense instead of $(\gamma, r)$-dense. 

Let $\mathsf{h}_{a_k}$ be the harmonic function introduced in \eqref{estim} for $\sigma=t$. 
Since $\mathsf{h}_{a_k}$ is harmonic on $D^t(a_k)$, there exists a function $H_{a_k}$ holomorphic on $D^t(a_k)$ with real part $\mathsf{h}_{a_k}$.

Now, given $f$ with $\|f\|_{L^p_{\phi}(\CC)}=1$, let $g=fe^{-\left( H_{a_k}+\phi(a_k)\right)}$ for $k\in I_f^{K-good}$, and set
\[
 h=c_0g, \quad c_0=\left(\frac{\pi s^2\rho(a_k)^2}{\int_{D^{s}(a_k)}|g|^pdA}\right)^{1/p}.
\]
Clearly there is $z_0\in D^{s}(a_k)$ with $|h(z_0)|\ge 1$. 

Set $R=2s$. We have to estimate the maximum modulus of $h$ on $D^R(a_k)$ in terms
of a local integral of $h$. To that purpose, we use that $h\in  A^p(D^{t}(a_k))$. Indeed, applying \eqref{estim} we have 
\beqa
\int_{D^{t}(a_k)}|h|^pdA
 &=& \frac{\pi s^2\rho(a_k)^2}{\int_{D^{s}(a_k)}|g|^pdA} \int_{D^{t}(a_k)}|g|^pdA\\
 &=& \frac{\pi s^2\rho(a_k)^2}{\int_{D^{s}(a_k)}|f|^p e^{-p(\mathsf{h}_{a_k}+\phi(a_k))}dA} \int_{D^{t}(a_k)}|f|^p e^{-p(\mathsf{h}_{a_k}+\phi(a_k))}dA\\
 &\le& \frac{\pi s^2\rho(a_k)^2 e^{2pA(t)}}{\int_{D^{s}(a_k)}|f|^pe^{-p\phi(z)}dA(z)} \int_{D^{t}(a_k)}|f|^pe^{-p\phi(z)}dA(z). \\
 &\le& \pi s^2\rho(a_k)^2 e^{2pA(t)} K^p
\eeqa
where the last inequality comes from the fact that $k$ is $K$-good for $(s,t)$. Therefore, taking into account this last estimate, the subharmonicity of $|h|^p$ yields
\begin{equation*}
 M^p:=\max_{z\in D^R(a_k)}|h(z)|^p \le \frac{1}{\pi s^2\rho(a_k)^2} \int_{D^{t}(a_k)}|h|^pdA \le  c_{s}K^p
\end{equation*}
where $c_{s}= e^{2pA(4s)}$ since $t=4s$. 

Now, setting $\tilde{E}=E\cap D^{s}(a_k)$ 
we get using Corollary \ref{CoroKov} applied to $h$:
\[
\int_{D^{s}(a_k)}|h(z)|^pdA(z) \\
 \le\left(\frac{cs^2\rho(a_k)^2}{|\tilde{E}|}\right)^{p\eta\log M+1}
 \int_{\tilde{E}}|h(z)|^pdA(z)\\
\]
Again, by homogeneity we can replace in the above inequality $h$ by $g$.
Note also that $\pi s^2\rho(a_k)^2/|\tilde{E}|$ is controlled by $1/\gamma$.
This yields

\beqa
 \int_{D^{s}(a_k)}|f|^pe^{-p\phi(z)}dA(z)&\le& e^{pA(t)}\int_{D^{s}(a_k)}|g(z)|^pdA(z)\\
&\le& e^{pA(t)}\left(\frac{cs^2\rho(a_k)^2}{|\tilde{E}|}\right)^{p\eta\log M+1}
 \int_{\tilde{E}}|g(z)|^pdA(z)\\
&\le& e^{pA(t)}\left(\frac{c_1}{\gamma}\right)^{p\eta\log M+1}\int_{\tilde{E}}|g(z)|^pdA(z)\\
&\le& e^{2pA(t)} \left(\frac{c_1}{\gamma}\right)^{p\eta\log M+1} \int_{\tilde{E}}|f(z)|^pe^{-p\phi(z)} dA(z),
\eeqa
where $c_1$ is an absolute constant.

Summing over all $K$-good $k$, and using Lemma \ref{lem1} and Proposition \ref{prop1} we obtain the required result
\[
  c \|f\|_{L^p_{\phi}(\CC)}\lesssim
 \left(\frac{c_1}{\gamma}\right)^{\eta\log M+1/p}\|f\|_{L_{\phi}^{p}(E)}
\]
where in view of Lemma \ref{Kovr-2D}
\[
\eta \leq c''\times 2^4 \log 2
\] 
and
\begin{align*}
\log(M)&\le \log(c_s^{1/p} K) =\log\left(\left(c_s \frac{N(4s)}{1-c}\right)^{\frac{1}{p}}\right) \le \frac{1}{p} \log\left(e^{2pA(4s)} c_{ov} \frac{(4s)^{\alpha}}{1-c} \right).
\end{align*}
Here $c$ comes from Proposition \ref{prop1}, and 
$c_{ov}$ and $\alpha$ from  Lemma \ref{lem1}. In particular, fixing \\
$c \in (0, \, 1)$ we obtain 
$$\log M \le 2A(4s) + \frac{1}{p}( C' + C'' \log(s))$$
where $C'$, $C''$ depend only on the space. 
In view of \eqref{eq-A-sigma}, we get 
$$\log M \lesssim s^{\log_2(C_\mu)} + \frac{1}{p}(1 +  \log(s)).$$ 
Finally, noticing that $r \asymp s=\max(r, r_0/2)$ for $r>1$, this implies 
$$\log M \lesssim r^{\log_2(C_\mu)} + \frac{1}{p}(1 +  \log(r)).$$
\end{proof}

\vspace{0.6cm}

Finally, we give a short proof of the reverse implication in Corollary \ref{cor1} which is a straightforward adaptation to the doubling Fock space of Luecking's proof of \cite[Corollary 3]{Lu1}.
\begin{proof}
First, assume that $\varphi \leq 1$. Therefore $s \leq 1$. 
Since $E$ is $(\gamma, r)$-dense, it is a dominating set. So 
$$\int_\CC \varphi^2 |f|^2 e^{-2 \phi} dA \geq s^2 \int_E |f|^2  e^{-2 \phi} dA \geq s^2 C^2 \|f\|^2_{L^2_{\phi}(\CC)},  $$
where $C$ is the sampling constant.  
Then
\begin{align*} 
\|(I-T_\varphi) f\|^2_{L^2_{\phi}(\CC)} &= \|T_{1-\varphi}f\|^2_{L^2_{\phi}(\CC)}\\
&= \|{\bf{P}}[(1-\varphi)f]\|^2_{L^2_{\phi}(\CC)} \\
&\leq \|(1-\varphi)f\|^2_{L^2_{\phi}(\CC)} \\
& \leq \int_\CC (1-\varphi^2)|f|^2 e^{-2\phi}dA \\
& \leq (1-s^2C^2)\|f\|^2_{L^2_{\phi}(\CC)}
\end{align*}
Hence $\| I-T_\varphi \| <1$. So $T_\varphi$ is invertible and 
$$\|T_\varphi^{-1}\|\leq \frac{1}{1- \| I-T_\varphi \|} \leq \frac{1}{1- \sqrt{1-s^2C^2}}.$$
Using Theorem \ref{thm1}, we obtain 
$$||T_\varphi^{-1}|| \le \frac{1}{1-\sqrt{1-s^2 \left(\frac{\gamma}{c}\right)^{2L}}}.$$ 
Finally, for a general $\varphi$, let $\psi = \frac{\varphi}{\| \varphi\|_{\infty}}$. Then $\psi \leq 1$, $E=\{\psi \geq \frac{s}{\| \varphi\|_{\infty}}=s'\}$ and $T_\varphi=\| \varphi\|_{\infty} T_\psi$. So, applying the previous discussion to $\psi$, it follows
$$\|T_\varphi^{-1}\| = \|T_\psi^{-1}\| \| \varphi\|^{-1}_{\infty} \le \frac{\| \varphi\|^{-1}_{\infty}}{1-\sqrt{1-(s')^2 \left(\frac{\gamma}{c}\right)^{2L}}} = \frac{\| \varphi\|^{-1}_{\infty}}{1-\sqrt{1-\left(\frac{s}{\|\varphi\|_{\infty}}\right)^2 \left(\frac{\gamma}{c}\right)^{2L}}}.$$
\end{proof}

\vspace{1cm}

\textbf{Acknowledgment.} This work is part of the first named author's thesis supervised
by Andreas Hartmann. The authors would like to thank him for many helpful discussions. 

\bibliographystyle{alpha}
\bibliography{biblio} 

\end{document}